\documentclass[12pt]{article}

\usepackage[letterpaper,margin=1in]{geometry}
\usepackage{epsfig,amsmath,amssymb,graphics,verbatim,amsfonts,subfigure,psfrag,amsthm}

\let\oldbibliography\thebibliography
\renewcommand{\thebibliography}[1]{%
  \oldbibliography{#1}%
  \setlength{\itemsep}{0.5mm}%
}

\newtheorem{thm}{Theorem}[section]
\newtheorem{cor}[thm]{Corollary}
\newtheorem{prop}[thm]{Proposition}
\newtheorem{lem}[thm]{Lemma}

\def\R{\mathbb{R}}

\def\P{\mathbb{P}}
\def\C{\mathbb{C}}
\def\N{\mathbb{N}}

\def\I{\infty}
\def\Id{{\textnormal{Id}}}

\def\txtd{{\textnormal{d}}}
\def\txte{{\textnormal{e}}}
\def\txti{{\textnormal{i}}}

\def\txtD{{\textnormal{D}}}

\newcommand{\be}{\begin{equation}}
\newcommand{\ee}{\end{equation}}
\newcommand{\bea}{\begin{eqnarray}}
\newcommand{\eea}{\end{eqnarray}}
\newcommand{\beann}{\begin{eqnarray*}}
\newcommand{\eeann}{\end{eqnarray*}}
\newcommand{\benn}{\begin{equation*}}
\newcommand{\eenn}{\end{equation*}}

\def\ra{\rightarrow}
\def\I{\infty}

\newcommand{\cC}{{\mathcal C}}  
\newcommand{\cF}{{\mathcal F}}  
\newcommand{\cG}{{\mathcal G}}  
\newcommand{\cI}{{\mathcal I}}  
\newcommand{\cK}{{\mathcal K}}  
\newcommand{\cL}{{\mathcal L}}  
\newcommand{\cO}{{\mathcal O}}  
\newcommand{\cT}{{\mathcal T}}  
\newcommand{\cX}{{\mathcal X}}  

\begin{document}

\title{Uncertainty Transformation via Hopf Bifurcation in 
Fast-Slow Systems}
\author{Christian Kuehn\thanks{Vienna University of Technology, Institute for Analysis and Scientific Computing, 1040 Vienna, Austria}}

\date{\today}   

\maketitle

\begin{abstract}
Propagation of uncertainty in dynamical systems is a significant challenge. Here we 
focus on random multiscale ordinary differential equation models. In particular, we 
study Hopf bifurcation in the fast subsystem for random initial conditions. 
We show that a random initial condition distribution can be transformed during 
the passage near a delayed/dynamic Hopf bifurcation: (I) to certain classes of symmetric 
copies, (II) to an almost deterministic output, (III) to a mixture distribution with 
differing moments, and (IV) to a very restricted class of general distributions. 
We prove under which conditions the cases (I)-(IV) occur in certain classes vector 
fields.
\end{abstract}

\medskip

\medskip

\section{Introduction}
 
Many mathematical models of complex systems contain an inherent element of uncertainty. 
From one perspective, it is a strength of theoretical models to abstract, simplify, 
and reduce a real system into a conceptual form. Modelling the neglected, unknown,
or different-scale processes can often be accomplished using probabilistic
models. The challenge is then to quantify uncertainty, i.e., to explain what effect 
random terms have in comparison to the purely deterministic system.

Here we study the scenario when we do not have exact information about the initial condition.
Suppose we model the initial condition as a random variable with a given distribution. Then the question 
is how the probability distribution is propagated by the dynamics? If the dynamical system 
contains an instability, e.g., a saddle-like structure in phase space, it is possible that 
a small random error in the initial condition can lead to widely different outcomes in the 
dynamics; see Figure~\ref{fig:01}(a). If all initial conditions in the 
distribution are attracted to a single stable attractor, then the randomness could probably have
been omitted in the first place; see Figure Figure~\ref{fig:01}(b). The critical cases are systems, 
which display transient and/or unstable behavior for a certain limited period of time in phase 
space; see Figure~\ref{fig:01}(c). This is the case considered in this paper.

\begin{figure}[htbp]
\psfrag{a}{(a)}
\psfrag{b}{(b)}
\psfrag{c}{(c)}
	\centering
		\includegraphics[width=0.75\textwidth]{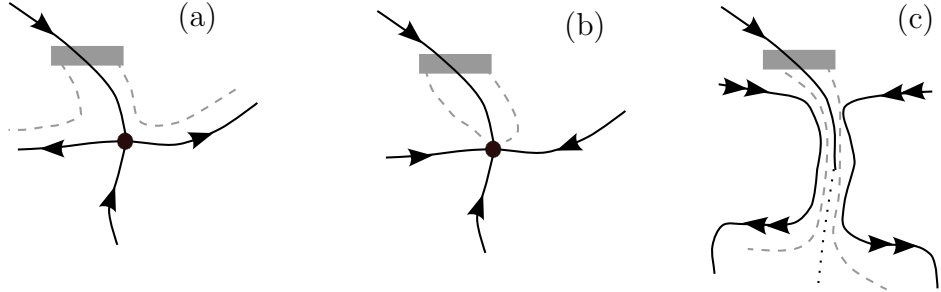}	
		\caption{\label{fig:01}Sketch of several situations including uncertain 
		initial conditions (grey rectangles). The flow is indicated with arrows, 
		steady states are marked as circles and propagation of certain 
		points in the set of potential initial conditions is given by dashed grey 
		curves. (a) Saddle steady state and exponentially diverging random initial 
		conditions. (b) Stable node and exponentially contracted random initial 
		conditions. (c) Transient instability.}
\end{figure} 

A passage near many instabilities is frequently modeled using multiple time scale 
dynamical systems. A fundamental subclass are fast-slow ordinary differential 
equations (ODEs). Slowly-drifting variables may bring the system towards an instability 
of certain fast variables. Near the instability an intricate interplay between 
the different classes of variables emerges. Detailed studies of many class of 
fast-slow bifurcation scenarios exist, see {e.g.}~\cite{DumortierRoussarie,KruSzm1,Schecter} 
and the extensive references in~\cite{KuehnBook}. The topic is 
sometimes referred to as 'dynamic bifurcation' or 'delayed bifurcation' and has a long 
history~\cite{Haberman,Shishkova,Su,DeMaesschalck,Benoit8,FruchardSchaefke}; for the 
delayed Hopf case considered in this
paper see~\cite{BaerErneuxRinzel,BaerGaekel,HoldenErneux,Neishtadt1,Neishtadt2,Stiefenhofer2}.

On the side of stochastic fast-slow systems the case of additive or multiplicative 
stochastic terms has been studied for multiscale stochastic ordinary differential 
equations (SODEs) from different perspectives~\cite{BerglundGentz,JansonsLythe} including 
bifurcation delay~\cite{Kuske,Kuske1,VandenBroeckMandel,ZeghlacheMandelVandenBroeck}; 
see~\cite[Sec.15.10]{KuehnBook} for more detailed references. The case of multiscale 
random ordinary differential equations (RODEs) has been explored a lot less up to now. 
Regarding delayed bifurcations of RODEs, a particular model case arising in pattern 
formation~\cite{ChenKololnikovTzouGai,TzouWardKolokolnikov} seems to be the first study. 

Here we concentrate on the abstract theory of delayed Hopf bifurcation for RODEs.
The Hopf case is definitely among the most interesting cases for bifurcation delay 
(see the review of results in Section~\ref{sec:detHopf} and the references mentioned above). 
The results of this work relate an input distribution $\mu_0$ of initial conditions to an 
output distribution $\mu_*$. We briefly state the conclusions in non-technical terms:

\begin{itemize}
 \item[(I)] The initial distribution $\mu_0$ can be transformed via only a 
 restricted set within the class of orthogonal transformations; certain reflections 
 are allowed (Theorem~\ref{thm:symmetry}) while general translations cannot occur 
 (Theorem~\ref{thm:shift}).
 \item[(II)] Given $\mu_0$ there exists a vector field mapping it to a 
 prescribed delta-distribution in a singular limit (Theorem~\ref{thm:buffer1})
 and the delta-distribution deforms to an approximate identity (Corollary~\ref{cor:det}).
 \item [(III)] For large classes of given real-analytic vector fields we obtain
 mixture measures for $\mu_*$ (Theorem~\ref{thm:mix}); the moments of $\mu_*$
 are computable in many cases (Proposition~\ref{prop:uniform}) and 
 (Proposition~\ref{prop:exp}).
 \item[(IV)] For general given $\mu_0$ and $\mu_*$, there is no
 real-analytic vector field which maps $\mu_0$ to $\mu_*$ under delayed 
 Hopf bifurcation (Theorem~\ref{thm:matrewrite} and 
 Proposition~\ref{prop:impossible}).
\end{itemize}  

In summary, we have shown that the problem of propagating uncertainty through
regions with bifurcations displays interesting behaviour, even for the simple
case of random initial condition and the codimension-one Hopf bifurcation. Although
there are many computational studies and approaches via inverse problems to
uncertainty quantification, the route via instabilities and multiscale bifurcation
normal form theory seems to be a wide-open direction for future work. 

\section{Deterministic Delayed Hopf Bifurcation}
\label{sec:detHopf}

We review basic results about deterministic delayed Hopf bifurcation to fix the 
notation and the setup. Consider a compact interval $\cI:=[0,\varepsilon_*]$ for 
some sufficiently small $\varepsilon_*>0$. Let $\varepsilon>0$, 
$\varepsilon\in \cI$, be a parameter representing the time scale 
separation. Consider the three-dimensional \emph{fast-slow system}
\be
\label{eq:fs}
\begin{array}{ccccl}
\varepsilon \frac{\txtd x_1}{\txtd \tau} &=& \varepsilon \dot{x}_1&=& f_1(x_1,x_2,y,\varepsilon),\\
\varepsilon \frac{\txtd x_2}{\txtd \tau} &=& \varepsilon \dot{x}_2&=& f_2(x_1,x_2,y,\varepsilon),\\
\frac{\txtd y}{\txtd \tau} &=& \dot{y}&=& g(x_1,x_2,y,\varepsilon),\\
\end{array}
\ee
where $f=(f_1,f_2)^\top:\R^3\times \cI\ra \R^2$, $g:\R^3\times \cI\ra \R$ are maps 
in a suitable function space $\cX$, $x=(x_1,x_2)^\top\in\R^2$ are the \emph{fast variables} 
and $y\in\R^1$ is the \emph{slow variable}. We also refer to \eqref{eq:fs} as a $(2,1)$-fast-slow 
system due the dimensions of the sets of variables. We restrict the analysis to 
suitable subsets of phase space with $x\in \cK_x\subset \R^2$, $y\in \cK_y\subset \R$,
where $\cK_x$ will always be compact.
We are going to need $\cX=C^k$ for some $k\in\N$ with $k\geq 3$, or $\cX=C^\I$, 
or $\cX=C^{\alpha}$ (real-analytic maps), depending on the setup; to avoid confusion 
with the probabilistic use of $\omega$ as an element of a sample space $\Omega$ we use 
the notation $\alpha$ as a superscript for real-analytic maps. In the notation we omit 
domain and range for function spaces, e.g., $f,g\in C^k$ is interpreted as 
$f\in C^k(\cK_x\times \cK_y\times \cI, \R^2)$ and in addition 
$g\in C^k(\cK_x\times \cK_y\times \cI, \R^1)$.

The system \eqref{eq:fs} is written on  the \emph{slow time scale} $\tau$ and can be 
re-written equivalently on the \emph{fast time scale} $t:=\tau/\varepsilon$. The 
\emph{critical manifold} of \eqref{eq:fs} is
\be
\cC_0:=\{(x,y)\in\cK_x\times \cK_y\subset \R^3:f(x,y,0)=0\}.
\ee 
$\cC_0$ can also be viewed as the algebraic constraint of the differential-algebraic \emph{slow 
subsystem} obtained from \eqref{eq:fs} by taking the singular limit $\varepsilon\ra 0$ which yields 
\be
\label{eq:ss}
\begin{array}{ccl}
0&=& f(x_1,x_2,y,0),\\
\dot{y}&=& g(x_1,x_2,y,0).\\
\end{array}
\ee
Alternatively, one may view $\cC_0$ as equilibrium points of the \emph{fast subsystem}
\be
\label{eq:fss}
\begin{array}{clccl}
\frac{\txtd x}{\txtd t}&=& x'&=&f(x_1,x_2,y,0),\\
\frac{\txtd y}{\txtd t}&=& y'&=& 0 .\\
\end{array}
\ee
obtained as a singular limit from \eqref{eq:fs} on the time scale $t$. For a more detailed 
introduction to multiple time scale dynamics see~\cite{KuehnBook}. The main assumptions 
for a generic \emph{delayed} (or \emph{dynamic}) \emph{Hopf bifurcation} to occur 
in~\eqref{eq:fs} are:

\begin{itemize}
 \item[(A1)] $\cC_0$ is a real-analytic one-dimensional curve and we 
 assume wlog that $\cC_0=\{(x,y)\in\cK_x\times\cK_y:x_1=0,x_2=0\}$;
 \item[(A2)] $\cC_0$ is \emph{normally hyperbolic} except at a single point, which we take without 
 loss of generality to be the origin $0:=(0,0,0)^\top \in\R^3$; more precisely, the matrix 
 $[\txtD_x f](p,0)=:J(p)\in\R^{2\times 2}$ has eigenvalues 
 $\lambda_1(p)=a_1(p)-\txti b_1(p),\lambda_2(p)=a_2(p)+\txti b_2(p)$ such that 
 $a_{1,2}(p)\neq 0$ for every $p\neq 0$, $p\in \cC_0$;
 \item[(A3)] at $p=0$ the fast subsystem has a Hopf bifurcation, {i.e.}, $\lambda_{1,2}(0)$
 is a complex conjugate pair of eigenvalues with nonzero imaginary part and we assume wlog 
 that $b_1(0)>0$ and $\textnormal{sign}(a_{1,2}(x_1,x_2,y))=\textnormal{sign}(y)$;
 \item[(A4)] the fast subsystem Hopf bifurcation at $p=0$ is generic, i.e., 
 $\frac{\txtd a_{1,2}}{\txtd y}(0)\neq 0$ and the first Lyapunov coefficient is nonzero;
 \item[(A5)] $g(0,0,y,0)\geq g_0>0$ for all $y\in \cK_y$ and some constant $g_0>0$.      
\end{itemize}

By (A1), we may write the slow subsystem \eqref{eq:ss} as 
\be
\label{eq:ss1}
\dot{y}=g(0,0,y,0).
\ee  
Denote the solution of \eqref{eq:ss1} with initial condition $y(\tau_0)=:y_0$ by $\xi(\tau)$.
The assumption (A5) guarantees that a trajectory crosses from the negative $y$-axis to the positive
$y$-axis. Up to a time translation, we may always assume for each individual slow subsystem 
trajectory that $\xi(0)=0$.

\begin{figure}[htbp]
\psfrag{x}{$x_1$}
\psfrag{y}{$y$}
\psfrag{Ca}{$\cC_0^a$}
\psfrag{Cr}{$\cC_0^r$}
\psfrag{KxKy}{$\cK_x\times \cK_y$}
\psfrag{y0}{\small{$(x_0,y_0)$}}
\psfrag{y1}{\small{$(x_*,y_*)$}}
\psfrag{Oe}{$\cO(\varepsilon)$}
\psfrag{g}{$\gamma_\varepsilon$}
	\centering
		\includegraphics[width=0.75\textwidth]{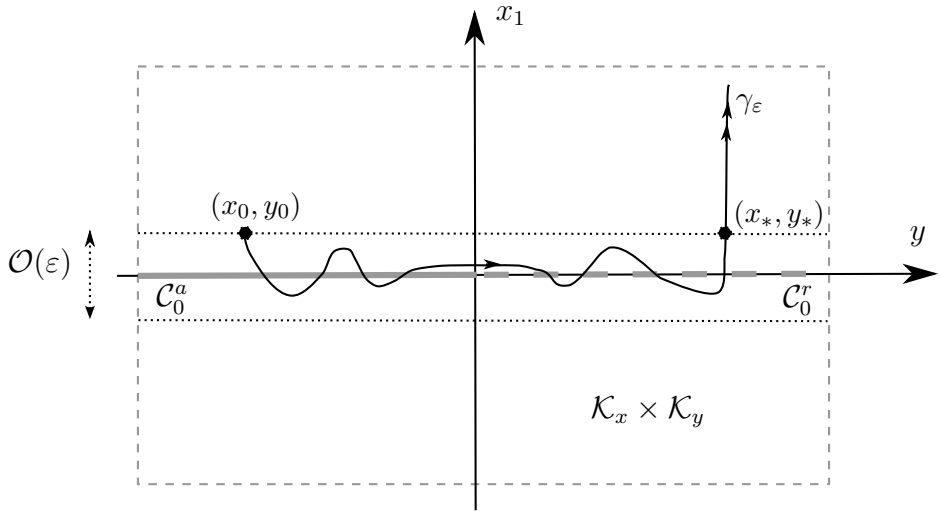}	
		\caption{\label{fig:02}Sketch of situation for a delayed Hopf bifurcation. 
		Projection onto $(y,x_1)$-coordinates. The domain $\cK_x\times \cK_y$ is 
		indicated as a dashed rectangle. A trajectory $\gamma_\varepsilon$ is
		shown (black curve) getting first attracted to and then repelled from $\cC_0$. 
		The initial condition $(x_0,y_0)=(x(\tau_0),y(\tau_0))$ is chosen so $\tau_0$ 
		is an asymptotic moment of falling while $(x_*,y_*)=(x(\tau_*),y(\tau_*))$ marks 
		the trajectory location for the asymptotic moment of jumping 
		$\tau_*$, where $\tau_*$ is also the delay time.}
\end{figure} 

Trajectories $\gamma_\varepsilon$ of the \emph{full fast-slow system} \eqref{eq:fs} with initial
conditions $y_0<0$, $y_0=\cO(1)$ as $\varepsilon\ra 0$, starting sufficiently close to the $y$-axis 
first get attracted towards 
\benn
\cC^a_0:=\{(x,y)\in\cK_x\times\cK_y:y<0\}\cap \cC_0. 
\eenn
Indeed, $\cC^a_0$ is normally hyperbolic \emph{attracting} since by (A3) we have negative real parts of the 
eigenvalues for the linearization, so Fenichel's Theorem \cite{Fenichel4,Jones} guarantees that 
there exists a \emph{slow manifold} $\cC^a_\varepsilon$ which is $\cO(\varepsilon)$-close to $\cC^a_0$ and 
the flow on $\cC^a_\varepsilon$ converges to the flow on $\cC^a_0$ as $\varepsilon\ra 0$; see also 
Figure~\ref{fig:02}. By~(A5), the slow dynamics on $\cC^a_\varepsilon$ guarantees that $\gamma_\varepsilon$ 
approaches a neighbourhood of the origin. The behavior of $\gamma_\varepsilon$ once it passes the Hopf 
bifurcation point and is near the \emph{repelling} part 
\benn
\cC_0^r:=\{(x,y)\in\cK_x\times\cK_y:y>0\}\cap \cC_0
\eenn
of the critical manifold is characterized by the following classical result:

\begin{thm}[\cite{Neishtadt1,Neishtadt2}]
\label{thm:neithm1}
Suppose (A1)-(A5) hold. Fix an initial time $\tau_0<0$. Assume 
$(x_1(\tau_0),x_2(\tau_0),y(\tau_0))$ is $\cO(\varepsilon)$-close to $\cC^a_0$ with associated 
trajectory $\gamma_\varepsilon(\tau)$. Then, there exists $\varepsilon_*>0$ 
such that for all $\varepsilon\in(0,\varepsilon_*]$, $\gamma_\varepsilon$ is in an 
$\cO(\varepsilon)$-neighborhood of $\cC^r_0$ for a delay time beyond the bifurcation point 
at $y=0$.
\begin{itemize}
 \item[(D1)] Suppose $f,g\in C^{\alpha}$ with complex analytic continuations in 
 the $(x,y)$-variables in a neighborhood of the origin remaining smooth with respect to $\varepsilon$. 
 Then $\gamma_\varepsilon$ has a delay time $\tau_*>0$ where $\tau_*=\cO(1)$ as $\varepsilon\ra 0$.
 \item[(D2)] Suppose $f,g\in C^\I$. Then the generic delay time is 
 $\sqrt{M(\varepsilon)\varepsilon|\ln\varepsilon|}$ where $M(\varepsilon)\ra +\I$ monotonically 
 as $\varepsilon\ra 0$.
 \item[(D3)] Suppose $f,g\in C^l$ for $l<\I$. Then the generic delay 
 time is of the order $\cO(\varepsilon|\ln\varepsilon|)$ as $\varepsilon\ra 0$.    
\end{itemize}  
\end{thm}    

The important distinction between cases (D1) and (D2)-(D3) is that a long delay is 
observed in the analytic case and a short delay in the remaining cases. The genericity 
requirement in cases (D2)-(D3) is necessary to guarantee that $\{y=0\}$ is no longer an 
invariant manifold for the full system when $\varepsilon>0$ and we shall make this assumption
from now on:
\begin{itemize}
 \item[(A6)] $\cC_0$ is not an invariant manifold for $\varepsilon\in(0,\varepsilon_*]$.
\end{itemize}
Furthermore, we are going to assume that $\cO(\cdot)$ is with respect to $\varepsilon$ and 
omit $\varepsilon\ra 0$ from now on.

For case (D1), calculating $\tau_*$ splits into two further cases. $\tau_*$ is called 
an \emph{asymptotic moment of jumping} of $\gamma_\varepsilon$ if in an 
$\cO(\varepsilon|\ln\varepsilon|)$-neighborhood of $\tau$, there is an interval 
$[\tau_a,\tau_b]$ such that $\gamma_\varepsilon(\tau_a)$ is $\cO(\varepsilon)$-close 
to $\cC_0$ and $\gamma(\tau_b)$ is $\cO(1)$ separated from $\cC_0$. 
$\tau_*$ is called an \emph{asymptotic moment of falling} if it is an asymptotic 
moment of jumping upon reversing time. Define the \emph{complex phase}
\benn
\Psi(\tau):=\int_{0}^{\tau} \lambda_1(\xi(s))~\txtd s.
\eenn 
Following~\cite{Neishtadt1,Neishtadt2}, we define the \emph{entry/exit-map} 
$\Pi:(-\I,0]\ra [0,+\I)$ by the requirement
\be
\label{eq:wayinwayout}
\text{Re}[\Psi(\tau)]=\text{Re}[\Psi(\Pi(\tau))].
\ee
Extending the domain of $\tau$ from $\R$ to $\C$, the pairs $(\tau,\Pi(\tau))$ can be 
connected by arcs 
\benn
\cL_k=\{\tau\in\C|\textnormal{ Re}[\Psi(\tau)]=k\}\subset \C.
\eenn
which are level sets of $\textnormal{Re}[\Psi(\tau)]$ for a given real number $k\in \R$. 
$\textnormal{Re}[\Psi(\tau)]$ is sometimes called the \emph{relief function}. Near $\tau=0$ the 
following conditions hold as consequences of~(A1)-(A5):

\begin{itemize}
 \item[(B1)] $\xi$ is analytic and $f_1,f_2,g$ are analytic at points of the slow flow solution;
 \item[(B2)] $\lambda_{1,2}(\xi(\tau))\neq 0$ and $\lambda_1(\xi(\tau))\neq \lambda_2(\xi(\tau))$;
 \item[(B3)] no tangent to the curves $\cL_k$ is vertical.
\end{itemize} 

(B1)-(B3) can fail far away from the Hopf bifurcation of the fast subsystem. Let 
$\tau^-$ and $\tau^+$ be the lower and upper bounds of endpoints of arcs $\cL_k$ for which 
(B1)-(B3) hold. Let $\Gamma$ be the arc starting at $\tau^-$ and ending at $\tau^+$ 
on which $\textnormal{Re}(\Psi(\tau))$ is constant. Denote the domain in the complex $z$-plane 
bounded by $\Gamma$ and its conjugate arc $\overline{\Gamma}$ by $\cG$.

\begin{thm}[\cite{Neishtadt1,Neishtadt2}]
\label{thm:neithm2}
Suppose $\tau_0\in(\tau^-,\tau^+)$ is an asymptotic moment of falling. Then $\Pi(\tau_0)$ 
is an asymptotic moment of jumping, and on the interval
\benn
(\tau_0+K\varepsilon|\ln\varepsilon|,\Pi(\tau_0)-K\varepsilon|\ln\varepsilon|),\quad  
\text{for some fixed constant $K>0$}
\eenn
the solution is $\cO(\varepsilon)$-close to $\cC_0$. If $\tau_0<\tau^-$ then the solution 
generically remains $\cO(\varepsilon)$-close to $\cC_0$ until $\tau<\tau^+-\delta(\varepsilon)$ 
where $\delta(\varepsilon)\ra 0$ as $\varepsilon\ra 0$. 
\end{thm} 

The time $\tau^+$ is called the \emph{buffer time} and $\xi(\tau^+)$ is the \emph{buffer 
point}. Theorem~\ref{thm:neithm2} states that there are two cases: either the integrated 
linearized variational contraction and expansion balance to determine the asymptotic 
moment of jumping, or all trajectories leave near the buffer point.

\section{Random Delayed Hopf Bifurcation} 
  
\subsection{Basic Setup}

Let $(\Omega,\cF,\P)$ be a probability space. Consider the random (2,1)-fast-slow system
\be
\label{eq:RDE}
\begin{array}{rcl}
\varepsilon \dot{x} & = & f(x,y,\varepsilon),\\
\dot{y} & = & g(x,y,\varepsilon),\\
\end{array}
\ee
with initial condition $(x(\tau_0),y(\tau_0))=(x_0(\omega),y_0(\omega))$. We are going 
to use as a solution concept for~\eqref{eq:RDE} \emph{sample function solutions}~\cite{Strand}. 
Suppose \eqref{eq:RDE} has a  delayed Hopf bifurcation satisfying 
assumptions (A1)-(A6) for every $\omega\in\Omega$. Next, one may divide the vector field by
$g$ due to the assumption (A5) and re-scale time to obtain
\be
\label{eq:RDE1}
\begin{array}{rcl}
\varepsilon \dot{x} & = & \tilde{f}(x,y,\varepsilon),\\
\dot{y} & = & 1,\\
\end{array}
\ee
where we are going to drop the tilde from $f$ in this section and focus on studying the
system~\eqref{eq:RDE1} satisfying (A1)-(A6) with initial condition 
$(x(\tau_0),y(\tau_0))=(x_0(\omega),y_0(\omega))$. This makes the slow subsystem
particularly simple so
\be
\label{eq:slow_simple}
\dot{y}=1, \quad \Rightarrow \quad y(\tau)=(\tau-\tau_0)+y_0(\omega).
\ee
We make the standard assumption that for each individual trajectory of the slow subsystem we require
$\xi(0)=0$, which implies that $\tau_0=\tau(\omega)$ becomes a random variable 
with the same distribution as $y_0(\omega)$. Of course, we can calculate from the 
distribution of the asymptotic moment of jumping $\tau_*(\omega)=\Pi(\tau(\omega))$ 
the distribution of $y_*(\omega):=y(\tau_*(\omega))$ just using~\eqref{eq:slow_simple}.

As a first step, we are only interested in the dynamic bifurcation effect in the slow 
coordinate $y$ and not in the precise location of the (oscillatory) fast variables $x$. 
So we take $(x(0),y(0))=(x_0,y_0(\omega))$ with
\be
\label{eq:mu0supp}
\P(y\leq y_0\leq y+\txtd y)=\mu_0(y_0),\qquad \textnormal{supp}(\mu_0)\subset (-\I,-\kappa_\mu],
\ee  
where $\mu_0$ is a probability measure and $\kappa_\mu>0$ is some given sufficiently 
small constant as we are not interested in initial conditions that do not undergo at 
least a certain delay. In particular, we can always make $\varepsilon$ sufficiently 
small to ensure that $x_0$ is not only deterministic but we also have $x(t)=\cO(\varepsilon)$ 
after a short transient time $t$ since $\cC_0^a$ is globally attracting for each fixed $y<0$. 
Therefore, we shall just assume $x_0=\cO(\varepsilon)$ and $(x_0,y_0(\omega))\not\in 
\cC_\varepsilon^a$ directly; see Figure~\ref{fig:02}. If $\mu_0$ admits a probability 
density, then we denote the density by $p_0$. The probability measure associated to 
$\tau_*$ will be denoted by $\mu_*$ and if it has a density we call it $p_*$.

\subsection{Orthogonal Transformation of Uncertainty}
\label{ssec:rigid}

The first situation we are going to analyze is what could be called simple passage (or orthogonal 
transformation) of uncertainty, i.e., of the probability density of initial conditions under delayed 
bifurcation by reflection and/or translation. 

\begin{thm}
\label{thm:symmetry}
Suppose $\mu_0$ has compact support. Then there exists $f\in C^\alpha$ such 
that (A1)-(A6) hold and $\mu_*(y)=\mu_0(-y)$.    
\end{thm}

\begin{proof}
Since $\mu_0$ has compact support it follows that we can restrict to $\cK_y$ 
compact. It suffices to find $f$ such that: (A1)-(A6) hold, the buffer time 
$\tau^+$ can be made large enough to move any buffer points outside 
$\cK_x\times \cK_y$, and $\tau_*=-\tau_0$. Consider 
\be
\label{eq:Hopf_nf}
f(x_1,x_2,y,\varepsilon)=\left(\begin{array}{l}cyx_1-x_2-x_1(x_1^2+x_2^2)                
\\ x_1+cyx_2-x_1(x_1^2+x_2^2)\\\end{array}\right)
    +\cO(\varepsilon)
\ee
for some constant $c>0$ to be chosen below, and select analytic higher-order 
$\cO(\varepsilon)$-terms such that $\cC_0\neq \cC_\varepsilon$ which yields (A6). 
Since $f$ is just to leading-order the normal form of a generic Hopf bifurcation 
with parameter $cy$, it easily follows that (A1)-(A4) 
are satisfied and (A5) trivially holds for $\dot{y}=1$. Next, we want to 
analyze the upper bounds on the buffer time imposed by the 
conditions (B1)-(B3). One calculates that $\lambda_{1,2}(s)=cs\mp\txti$ so 
\benn
\Psi(\tau)=\int_{0}^{\tau} cs-\txti~\txtd s=\frac{c}{2}\tau^2-\txti \tau.
\eenn  
Hence, letting $\tau=u+\txti v$ one gets 
\benn
\textnormal{Re}(\Psi(\tau))=\frac{c}{2}\left(u^2-v^2\right)+v:=U(u,v).
\eenn 
We start with the upper bound imposed by (B3). Vertical tangents to the 
level sets $\cL_k=\{U(u,v)=k\}$ appear if $\frac{\partial U}{\partial v}=-cv+1=0$, 
i.e., for $v=1/c$. Level curves connecting from the vertical tangency point 
$(u_1,1/c)$ to a point $(u_2,0)$ have to satisfy 
\benn
\frac{c}{2}\left(u_1^2-c^{-2}\right)+\frac1c=k,\qquad \text{and}\qquad 
\frac{c}{2}u_2^2=k.
\eenn
The level curves $\cL_k$ delimiting $\cG$ in the upper half-plane are given by 
$k=1/(2c)$ so $u_2=\pm 1/c$. This implies that (B3) yields an upper bound
on the buffer time given by $1/c$. Regarding (B2), it is easy to see that 
$\lambda_1\neq \lambda_2$ and $\lambda_1(s)=0$ if and only if $s=\txti/c$ or $v=1/c$. 
Therefore, (B2) leads to the same upper bound on the buffer time as the condition 
(B3). Lastly, (B1) does not yield any upper bound on $\tau^+$. Therefore, we find 
$\tau^+=1/c$ and we can move any buffer
point outside of a given compact region $\cK_x\times \cK_y$ upon decreasing $c$. 
Lastly, we have to check that $\tau_*=-\tau_0$. Again one may just calculate that
\benn
\textnormal{Re}(\Psi(\tau_0))=\textnormal{Re}(\Psi(\tau_*))\quad \Leftrightarrow\quad
\frac{c}{2}\tau_0^2=\frac{c}{2}\tau_*^2
\eenn
since the start and end times must be real-valued. Upon using that $\tau_*>\tau_0$ 
and $c>0$ we may conclude that $\tau_*=-\tau_0$.
\end{proof}

\begin{figure}[htbp]
\psfrag{a}{(a)}
\psfrag{b}{(b)}
\psfrag{py}{$\mu_{\cdot}(y)$}
\psfrag{y}{$y$}
\psfrag{m0}{$\mu_{0}$}
\psfrag{ms}{$\mu_*$}
\psfrag{m}{$m$}
	\centering
		\includegraphics[width=0.85\textwidth]{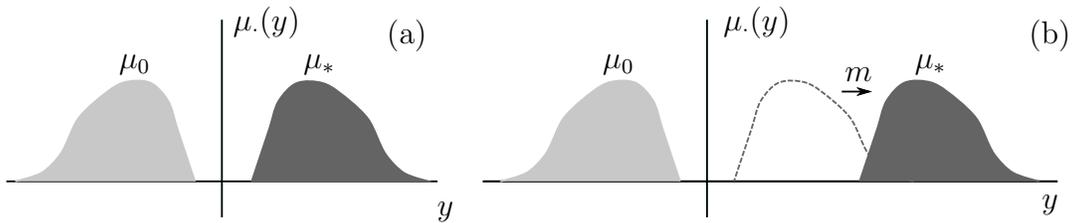}	
		\caption{\label{fig:03}Sketch of the situation for orthogonal 
		transformations. (a) Illustration of Theorem~\ref{thm:symmetry}
		where $\mu_*(y)=\mu_0(-y)$ is obtained via reflection. The initial
		density is shown in light grey and the transported one in dark
		grey. (b) The situation involving reflection (dashed density) and
		a nontrivial shift by $m\neq 0$ is not possible for analytic vector 
		fields according to Theorem~\ref{thm:shift}.}
\end{figure} 

The proof of Theorem~\ref{thm:symmetry} carries over verbatim if higher-order 
nonlinear perturbations are added to the Hopf normal form~\eqref{thm:symmetry}; 
see Figure~\ref{fig:03}(a) for an illustration.
However, not all natural transformations preserving the shape of the $\mu_0$ are 
allowed as the next result, quite surprisingly, shows.

\begin{thm}
\label{thm:shift}
Fix $m\in\R-\{0\}$, suppose $\mu_0$ has compact support 
$\textnormal{supp}(\mu_0)=\cI_\mu$ containing at least one accumulation point, 
and $\textnormal{supp}(\mu_0(\cdot+m))\subset(-\I,0)$. Then there does \emph{not} 
exist $f\in C^\alpha$ such that (A1)-(A6), (B1)-(B3) hold and 
\benn
\mu_*(y)=\mu_0(-y+m).  
\eenn
\end{thm}

\begin{proof}
We argue by contradiction and suppose that an analytic vector field $f$ exists 
satisfying (A1)-(A6), (B1)-(B3) such that $\mu_*(y)=\mu_0(-y+m)$ for some positive 
$m>0$; note that this situation corresponds to reflecting and shifting the initial 
condition $y$-distribution. Using (B1)-(B3) and Theorem~\ref{thm:neithm2} for the 
case of no buffer points it follows that
\benn
\textnormal{Re}(\Psi(-\tau_0+m))-\textnormal{Re}(\Psi(\tau_0))=0
\eenn
for all $\tau_0\in\cI_\mu$. As above let $\lambda_1(s)$ denote the eigenvalue in 
the definition of $\Psi$. Since (B2) is always assumed to hold independent of the 
point $p$, it follows that the discriminant of the Jacobian $J(p)$ does not vanish. 
The discriminant must be negative to get complex conjugate eigenvalues. Therefore, 
$\lambda_1(s)=a_1(s)-\txti b_1(s)$ where $a_1(s)$ is real-analytic as a function 
of $s\in\R$ and 
\benn 
\textnormal{Re}(\Psi(-\tau_0+m))-\textnormal{Re}(\Psi(\tau_0))=
\int_{\tau_0}^{-\tau_0+m} a_1(s)~\txtd s=:A_1(\tau_0).
\eenn
Since $a_1$ is real-analytic it follows that $A_1(\tau_0)$ is real-analytic. 
$A_1(\tau_0)$ vanishes on $\cI_\mu$ which contains an accumulation point. 
Extending $A_1(\tau_0)$ to a sufficiently small neighbourhood of $\cI_\mu$ 
into the complex plane we may apply the principle of permanence and conclude that 
$A_1(\tau_0)$ vanishes also at $\tau_0=0$. This implies 
\benn
\int_{0}^{m} a_1(s)~\txtd s=0
\eenn
and so, since $m>0$ and $a_1(s)>0$ for $s>0$ by (B2), we obtain a contradiction.
\end{proof}

The last result shows that there is some rigidity in the way uncertainty can 
be transported for analytic vector fields without buffer points and suitable 
uniform eigenvalue configurations; see also Figure~\ref{fig:03}(b) and 
Section~\ref{ssec:transform}. 

\subsection{Random-to-Deterministic Mapping}
\label{ssec:randdet}

In Section~\ref{ssec:rigid} we have considered the case when uncertainty 
gets just mapped via orthogonal transformations (translation, reflection). 
In this section, we address the other extreme case and study when we obtain 
a deterministic, or at least almost deterministic, output after passing a 
delayed bifurcation.

\begin{prop}
\label{prop:buffer}
Given $\mu_0$ satisfying \eqref{eq:mu0supp}, there 
exists $f\in C^\alpha$ such that (A1)-(A6) hold and $\mu_*(y)=
\delta_{y^+}(y)$, i.e., the output is a delta-distribution located at 
some $y^+>0$.    
\end{prop}

\begin{proof}
As in the proof of Theorem~\ref{thm:symmetry} we select~\eqref{eq:Hopf_nf}. 
By~\eqref{eq:mu0supp} the support of $\mu_0$ is contained in $(-\I,-\kappa_\mu]$ 
for some fixed positive $\kappa_\mu>0$. Recall from the proof of 
Theorem~\ref{thm:symmetry} that the buffer point is given by $\tau^+=1/c$, 
where $c>0$ is the parameter in the vector field~\eqref{eq:Hopf_nf}. Making $c$ 
sufficiently small we can guarantee that $\tau_*>1/c$ so all trajectories with 
initial conditions sampled from $\mu_0$ jump at the buffer point.
\end{proof}

Proposition~\ref{prop:buffer} shows that in the singular limit $\varepsilon=0$ any 
uncertainty in the initial condition disappears. In fact, one may even do 
slightly better for initial distributions with compact support and exhibit examples 
for any target $\delta$-distribution.

\begin{thm}
\label{thm:buffer1}
Let $\mu_0$ have compact support and satisfy~\eqref{eq:mu0supp}. Furthermore, fix any 
$y^+>0$. Then there exists an analytic vector field $f$ satisfying (A1)-(A6) such that 
\benn
\mu_*(y)=\delta_{y^+}(y).
\eenn
\end{thm}

\begin{proof}
Consider the modified Hopf normal form 
\be
\label{eq:Hopf_nf1}
f(x_1,x_2,y,\varepsilon)=\left(\begin{array}{l}\txte^{-ay}y x_1-bx_2-x_1(x_1^2+x_2^2)                
\\ b x_1+\txte^{-ay}y x_2-x_1(x_1^2+x_2^2)\\\end{array}\right)
    +\cO(\varepsilon)
\ee
with parameters $a,b>0$ and suitable higher-order analytic 
$\cO(\varepsilon)$-terms such that $\cC_0\neq \cC_\varepsilon$. A direct
calculation shows that $\lambda_{1,2}(s)=\txte^{-as}s\mp b\txti$. First, 
we are going to investigate the locations of buffer points. One has for 
$\tau=u+\txti v$ with $u,v\in\R$ that
\benn
\textnormal{Re}(\Psi(\tau))=bv-\frac{e^{-a u} ((a u+1) 
\cos (a v)+a v \sin (a v))}{a^2}.
\eenn
Upon increasing $a>0$ we can ensure that (B3) is satisfied in a region $\cG$ 
delimited by two arcs $\Gamma$ and $\bar{\Gamma}$ such that $\cG$ is inside 
a region 
\benn
\{u+\txti v\in \C:u\in[-u^-,u^+),[-v^-,v^+]\}
\eenn
where $u^-,v^-,v^+>0$ are fixed. Since the level curves of 
$\textnormal{Re}(\Psi(\tau))$ become almost horizontal in the positive 
half-plane in the limit $a\ra +\I$, it follows that $u^+\ra +\I$ 
as $a\ra +\I$. Therefore, (B3) gives no upper bound for the buffer 
time $\tau^+$ if $a$ is sufficiently large. Clearly, (B1) always holds. 
Regarding (B2), observe that $\lambda_1(s^+)=0$ if and only if 
\be
\label{eq:transcendental}
\txti b=\txte^{-as^+}s^+.
\ee
The solutions $s^+=s^+(b)$ of the transcendental equation~\eqref{eq:transcendental} 
for $s^+\in\C$ satisfy 
\benn
\lim_{b\ra 0}|s^+(b)|=0 \qquad \text{and}\qquad \lim_{b\ra +\I}|s^+(b)|=+\I
\eenn
for fixed $a>0$. Hence, once $a>0$ has been fixed we can use $b>0$ to get any 
buffer time $\tau^+$. Next, we show that $a>0$ can indeed be chosen and fixed 
so that any trajectory with initial condition sampled from $\mu_0$ does reach the 
buffer time before escaping. One has for real values of $\tau$ and $\tau_0$ that
\benn
\textnormal{Re}(\Psi(\tau_*))-\textnormal{Re}(\Psi(\tau_0))=0 \quad \Leftrightarrow \quad 
(a \tau_* +1) e^{a \tau_0}=e^{a \tau_* } (a \tau_0+1)
\eenn
Since $\mu_0$ has compact support contained in $(-\I,-\kappa_\mu]$ for some 
$\kappa_\mu>0$ there exists a large $a>0$ such that $\tau_*=+\I$ for all 
$\tau_0\leq -\kappa_\mu$. Therefore, if we select $a>0$ sufficiently large and 
then select $b>0$ we can achieve any finite buffer time. All trajectories 
sampled from $\mu_0$ jump at this prescribed buffer time, respectively the prescribed 
buffer point.
\end{proof}

The main insight in the last proof is that an asymmetric strength of the attracting 
and repelling eigenvalues can be used to make the final jump time $\tau_*$ calculated 
from entry-exit map large using one parameter. The second parameter is then used 
to create and move a buffer point $\tau^+$ which leads to an escape; see also 
Figure~\ref{fig:04}(a) for an illustration of Theorem~\ref{thm:buffer1}.\medskip

\begin{figure}[htbp]
\psfrag{a}{(a)}
\psfrag{b}{(b)}
\psfrag{py}{$\mu_{\cdot}(y)$}
\psfrag{y}{$y$}
\psfrag{y+}{$y^+$}
\psfrag{m0}{$\mu_{0}$}
\psfrag{ms}{$\mu_*$}
\psfrag{mse}{$\mu_*^\varepsilon$}
	\centering
		\includegraphics[width=0.85\textwidth]{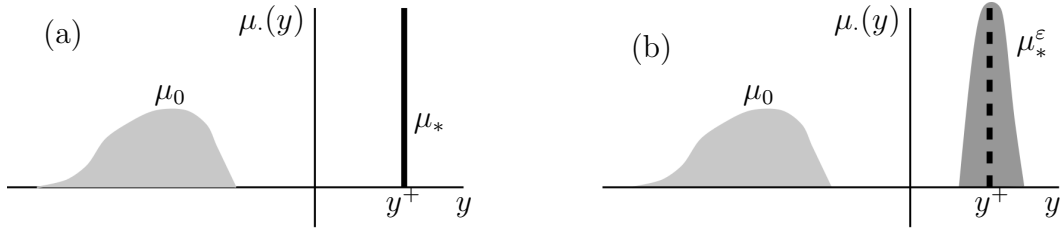}	
		\caption{\label{fig:04}Sketch of the transformation to an almost 
		deterministic output. (a) Illustration of Theorem~\ref{thm:buffer1}
		in the singular limit $\varepsilon=0$. The initial density is shown 
		in light grey and the transported delta mass as a solid black bar. (b) 
		The same situation as in (a) just for $0<\varepsilon\ll1$ as stated
		in Corollary~\ref{cor:det}; the dashed line indicates the singular limit
		distribution and the dark grey is the transformed density.}
\end{figure}

It is clear that the results in Proposition~\ref{prop:buffer} and 
Theorem~\ref{thm:buffer1} are not quite what would be observed in practice 
in a numerical or experimental setup as one has to consider $0< \varepsilon \ll1$ 
instead of $\varepsilon=0$ as shown in Figure~\ref{fig:04}(b). To analyze this 
case we need some preliminary considerations. Define the set
\benn
\cT(h):=\{(x,y)\in\R^3:\|x\|\leq h\varepsilon\}
\eenn
and fix $h>0$ always so that 
\benn
\cC^a_\epsilon,\cC^r_\epsilon\subset \cT(h)\qquad \text{and}\qquad 
\partial \cT(h)\cap (\cC^a_\epsilon\cup\cC^r_\epsilon)=\emptyset,
\eenn
where $\partial \cT(h)$ denotes the boundary of $\cT(h)$.
Let $p^\varepsilon_0(y)$ be a probability density of initial 
$y$-coordinate conditions with the same support condition as 
in~\eqref{eq:mu0supp} and fix some $x_0$ such that $(x_0,y_0(\omega))
\in\partial\cT(h)$, where $y_0(\omega)$ is sampled from $p_0^\varepsilon(y)$. 
Let $y_*(\omega)$ denote the $y$-coordinate of the point in $\partial \cT(h)$, 
where a trajectory of~\eqref{eq:RDE1} starting at $y_0(\omega)$ first 
leaves $\cT(h)$. Denote the associated probability density of $y_*(\omega)$ 
by $p_*^\varepsilon(y)$.

\begin{lem}
\label{lem:smooth}
Suppose $f,p_0^\varepsilon\in C^r$ for some $r\in\N_0$, $r=\I$, or $r=\alpha$, 
then $p_*^\varepsilon\in C^r$.
\end{lem}

\begin{proof}
The result for $r\in\N_0$ and $r=\I$ follows from the classical theory of 
continuous/differentiable and smooth dependence of solutions of ODEs on initial 
conditions~\cite{CoddingtonLevinson}. In fact, one may also prove that for analytic
vector fields solutions depend analytically on initial data~\cite[Sec.C.3]{Sontag}.
\end{proof}

\begin{cor}
\label{cor:det}
Suppose $f\in C^\alpha$, $p_0^\varepsilon\in C^\I$ for all 
$\varepsilon\in[0,\varepsilon_*]$ for some
sufficiently small $\varepsilon_*>0$, and $p_0^\varepsilon$ has compact support 
satisfying~\eqref{eq:mu0supp}. Furthermore, fix any $y^+>0$. Then there exists 
an analytic vector field $f$ satisfying (A1)-(A6) such that
\be
\label{eq:nascent}
\lim_{\varepsilon\ra 0}\int_{-\I}^\I p_0^\varepsilon(y)w(y)~\txtd y = w(y^+)
\ee
for all $w\in C^\I_c$ (smooth functions with compact support). 
\end{cor}

\begin{proof}
This is just a combination of Theorem~\ref{thm:neithm2}, Theorem~\ref{thm:buffer1}, 
and Lemma~\ref{lem:smooth}.
\end{proof}

Therefore, we observe in practice the convergence to a $\delta$-distribution
via an approximation to the identity. It is clear that also a modification of 
Corollary~\ref{cor:det} holds which uses the assumptions of Proposition~\ref{prop:buffer}
instead the ones from Theorem~\ref{thm:buffer1}.

\subsection{Mixtures}
\label{ssec:mixures}

The case of purely deterministic output distribution $\mu_*=\delta_{\tau^+}$ is already very 
interesting for applications. In this section, we study the case when $\mu_0$ 
is neither symmetrically reflected nor mapped to a completely deterministic distribution 
$\mu_*$ as shown in Figure~\ref{fig:05}.

\begin{thm}
\label{thm:mix}
Suppose $f\in C^\alpha$ and assumptions (A1)-(A6), as well as (B1)-(B3) hold
up to a given buffer time $\tau_+\in(0,\I)$. Then $\mu_*$ is a mixture measure
\be
\label{eq:mix}
\mu_*=\rho_{*,1}\delta_{\tau_+}+\rho_{*,2}\mu_{*,2},
\ee
for some $\rho_{*,1}+\rho_{*,2}=1$ and a probability measure $\mu_{*,2}$.
If $\tau_*=-\tau_0$ for $\tau_*<\tau_+$ holds and $\mu_0$ satisfying~\eqref{eq:mu0supp} 
has density $p_0$ then
\be
\label{eq:mixspecial}
\rho_{*,1}=\int_{-\I}^{\tau_-}\txtd \mu_0(s),\quad
\rho_{*,2}=\int_{\tau_-}^0\txtd \mu_0 (s),\quad 
\txtd \mu_{*,2}(s)={\bf 1}_{\{0\leq s<\tau_+\}}~\txtd \mu_0(-s).
\ee
\end{thm}

\begin{proof}
The result~\eqref{eq:mix} follows from the existence of a buffer time $\tau_+$ since 
all times $\tau<\tau_-$ satisfy $\Pi(\tau)=\tau_+$ yielding a delta-distribution at
$\tau_+$. The special case~\eqref{eq:mixspecial} holds since $-\tau_0=\tau_*$ for 
times not reaching the buffer time yields $\mu_*(s)=\mu_0(-s)$, and the weights $\rho_{*,1}$
and $\rho_{*,2}$ are just computed from the probability which points reach the buffer 
time and which ones do not.
\end{proof}

Although the assumption $-\tau_0=\tau_*$ seems quite special, it should be noted
that this is precisely the situation which happens for the standard Hopf bifurcation
normal norm; {cf.}~Theorem~\ref{thm:symmetry} and the vector 
field~\eqref{eq:Hopf_nf}.\medskip

\begin{figure}[htbp]
\psfrag{a}{(a)}
\psfrag{b}{(b)}
\psfrag{py}{$\mu_{\cdot}(y)$}
\psfrag{y}{$y$}
\psfrag{m0}{$\mu_{0}$}
\psfrag{ms}{$\mu_*$}
\psfrag{mse}{$\mu_*^\varepsilon$}
	\centering
		\includegraphics[width=0.85\textwidth]{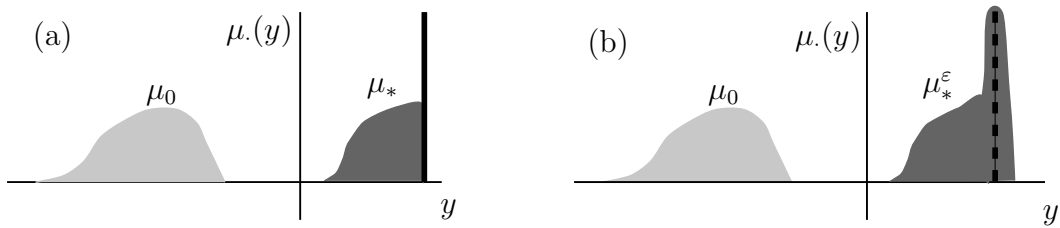}	
		\caption{\label{fig:05}Sketch of the mixture situation from 
		Theorem~\ref{thm:mix}. The notational and grey-shading conventions are 
		as in Figures~\ref{fig:03}-\ref{fig:04}. (a) Singular limit 
		$\varepsilon=0$ leading to the mixture of a delta-mass at the 
		buffer point and a remaining part calculated via the entry/exit-map. 
		(b) The same situation as in (a) just for $0<\varepsilon\ll1$.}
\end{figure} 

It is interesting to compute with a few classical initial
time distributions for the Hopf normal form case. To simplify the notation we are 
going to define
\be
m_{0,q}:=\int_{-\I}^\I s^q~\txtd \mu_0(s),\qquad m_{*,q}:=\int_{-\I}^{\I}s^q~\txtd \mu_*(s),
\ee
for $q\in\N$ as the $q$-th moments of $\mu_0$ and $\mu_*$. We are going to assume
that the moments do exist.

\begin{cor}
\label{cor:pform}
Suppose $f\in C^\alpha$ and satisfies assumptions (A1)-(A6), as well as (B1)-(B3) 
up to a given buffer time $\tau_+\in(0,\I)$. Furthermore, assume $\tau_*=-\tau_0$ for
$\tau_*<\tau_+$ holds and $\mu_0$ satisfying~\eqref{eq:mu0supp} has density $p_0$ then
\be
\label{eq:mq}
m_{*,q}=\left(\int_{\tau_-}^0 p_0(s)~\txtd s\right)
\left(\int_0^{\tau_+} s^q~p_0(-s)~\txtd s\right)
+\tau_+^q \int_{-\I}^{\tau_-}~p_0(s)~\txtd s.
\ee
\end{cor}

\begin{prop}
\label{prop:uniform}
Suppose $p_0$ is a uniform density with support in $[-b,-a]$ for $b>a>0$ and the assumptions 
of Corollary~\ref{cor:pform} hold. Then three cases occur 
\begin{itemize}
 \item[(U1)] $a\geq \tau_+$: $m_{*,q}=(\tau_+)^q$,
 \item[(U2)] $a< \tau_+\leq b$: $m_{*,q}=\frac{(b-\tau_+)(\tau_+)^q}{b-a}
 +\frac{(\tau_+-a)((\tau_+)^{q+1}-a^{q+1})}{(b-a)(q+1)}$,
 \item[(U3)] $b\leq \tau_+$: $m_{*,q}=\frac{b^{q+1}-a^{q+1}}{(b-a)(q+1)}$.
\end{itemize}
\end{prop}

\begin{proof}
Only the case (U2) is interesting, the other two cases are trivial. 
Since $\tau_*=-\tau_0$ for $\tau_0>\tau_-$ we also must have $-\tau_-=\tau_+$. Therefore,
we find
\benn
\rho_{*,1}=\int_{-b}^{-\tau_+}\frac{1}{b-a}\txtd s=\frac{b-\tau_+}{b-a},\quad
\rho_{*,2}=\int_{-\tau_+}^{-a}\frac{1}{b-a}\txtd s=\frac{\tau_+-a}{b-a},
\eenn
and so using Theorem~\ref{thm:mix} and calculating
\benn
\int_0^{\tau_+} s^q~p_0(-s)~\txtd s=\frac{b^{q+1}-a^{q+1}}{(b-a)(q+1)}
\eenn
yields the result.
\end{proof}

We also refer to cases, such as (U2), to the \emph{full mixture} case, i.e., when 
$\rho_{*,1}\neq 0$ and $\rho_{*,2}\neq 0$. If $\rho_{*,1}= 0$ or $\rho_{*,2}=0$,
such as for (U1) or (U3), we refer to the situation as \emph{singular mixture}.

\begin{prop}
\label{prop:exp}
Suppose $p_0$ is an exponential density with support in $(-\I,-a]$ for $a>0$ with 
rate $\beta^{-1}$, i.e., $p_0(s)= \beta^{-1}\txte^{(s+1)/\beta}$ and the assumptions 
of Corollary~\ref{cor:pform} hold. Then two cases occur 
\begin{itemize}
 \item[(E1)] $a\geq \tau_+$: $m_{*,q}=(\tau_+)^q$,
 \item[(E2)] $a< \tau_+$: let $\Gamma(z_1,z_2)=\int_{z_2}^\I s^{z_1-1}\txte^{-s}~\txtd s$ be 
 the incomplete gamma-function, then
 \benn
 m_{*,q}=\txte^{(a-\tau_+)/\beta}(\tau_+)^q+\txte^{1/\beta}\beta^q
 (1-\txte^{(a-\tau_+)/\beta}) 
 \left(\Gamma \left(q+1,\frac{a}{\beta }\right)-
 \Gamma \left(q+1,\frac{\tau_+}{\beta }\right)\right).
 \eenn
\end{itemize}
\end{prop}

\begin{proof}
The only minor difference to the type of calculation in the proof of 
Proposition~\ref{prop:uniform} is that the integral $\int_0^{\tau_+} s^q~p_0(-s)~\txtd s$
is slightly more complicated and can be easily re-written in terms of 
incomplete gamma-functions.
\end{proof}

In principle, one can now also do many other types of calculations for given
classical probability measures $\mu_0$. Furthermore, similar smooth approximation
results as Corollary~\ref{cor:det} hold for mixture cases but we are not going 
to state them explicitly here.

\subsection{Distribution Transformation}
\label{ssec:transform}

In addition to rigid transformations, random-to-deterministic mappings, and mixture
distributions, one may also ask under which conditions we could obtain a particular
target distribution. The questions is, given two probability measures 
$\mu_\textnormal{in}$ and $\mu_\textnormal{out}$, does there exists a vector field 
such that $\mu_0=\mu_\textnormal{in}$ and $\mu_\textnormal{out}=\mu_*$? We have already
seen in Theorem~\ref{thm:shift} that this question is nontrivial. We make 
the following assumptions for the problem setup:

\begin{itemize}
 \item[(M1)] suppose $\mu_0$ is given and has density $p_0$,
 \item[(M2)] there exists a mapping $\Pi$, which is invertible and 
 analytic on $\textnormal{supp}(p_0)$, such that $\Pi(\tau_0)=\tau_*$. 
\end{itemize}

Assumption~(M1) on the existence of a density is made for notational convenience; it 
can be slightly relaxed to the accumulation point condition as in 
Theorem~\ref{thm:shift}. Assumption~(M2) implies 
$p_*(r)=(p_0\circ \Pi^{-1})(r) \frac{\txtd \Pi^{-1}}{\txtd r}(r)$. Dropping analyticity 
in~(M2) leads to a relatively trivial problem as
Theorem~\ref{thm:neithm1} implies that the number of target measures $\mu_*$ is 
extremely restricted in the case $f\not\in C^{\alpha}$. We are going to 
use the notation 
\be
\Pi(s)=\sum_{k=0}^\I \pi_ks^k
\ee
for the convergent power series of the map $\Pi$. We are going to show that for certain 
classes of given analytic maps $\Pi$ a certain necessary condition for the existence 
of an analytic vector field $f$, satisfying (A1)-(A6) and (B1)-(B3), can be based upon on the 
classical theory of infinite matrices. For infinite-dimensional matrices with countable 
indices we use the notation $M=(m_{ij})_{i,j\in \N}$ and denote by $\Id$ the matrix 
with entries given by the Kronecker delta $\delta_{ij}$. $M$ is lower-diagonal
if $m_{ij}=0$ for all $j>i$.

\begin{thm}
\label{thm:matrewrite}
A necessary condition for the existence of $f\in C^\alpha$ without buffer 
points such that (M1)-(M2), (A1)-(A6), and (B1)-(B3) hold is the existence of an
infinite vector $v:=(v_1,v_2,\ldots)^\top$ with $v_k\in\R$ such 
that $v_k\neq 0$ for some $k\in\N$, and $v$ satisfies a linear system 
\be
\label{eq:linsyst}
(M- \textnormal{Id})v=0 
\ee
for a matrix $M$, and $M$ is computable recursively from $\{\pi_k\}_{k\in\N_0}$
if $\pi_0\neq 0$.
\end{thm}

\begin{proof}
As in the proof of Theorem~\ref{thm:shift}, we know that the discriminant of 
the Jacobian along the critical manifold is negative so that $\lambda_1=a_1+\txti b_1$
where $a_1=a_1(s)$ is real-analytic. Furthermore, since there are no buffer points, 
we must have
\be
\label{eq:condmat}
0=\textnormal{Re}(\Psi(\tau_*))-\textnormal{Re}(\Psi(\tau_0))=
\int_{\tau_0}^{\tau_*} a_1(s)~\txtd s=A_1(\tau_*)-A_1(\tau_0) 
\ee
where $A_1$ is obtained via term-by-term integration of $a_1$. Clearly, $A_1$
is real-analytic as well and we use the notation $A_1(s)=\sum_{k=2}^\I v_{k-1}s^k$; note
the vanishing of the eigenvalue $a_1(0)=0$ in~(A3) implies the particular form
of the power series of $A_1$. Clearly, \eqref{eq:condmat} is a necessary 
condition for the existence of $f$, which can be re-written as
\bea
A_1(\tau_*)-A_1(\tau_0)&=&A_1(\Pi(\tau_0))-A_1(\tau_0)\nonumber\\
&=&\sum_{k=2}^\I v_{k-1}\left(\sum_{j=0}^\I\pi_j s^j\right)^k-
\sum_{k=2}^\I v_{k-1}s^k=0\label{eq:condser}.
\eea
Since $\mu_0$ has a density $p_0$, it follows that the domain of $\lambda_1$, 
and hence the domain of $a_1$ and $A_1$, has an accumulation point. Therefore, 
the equality~\eqref{eq:condser} holds if and only if the coefficients of 
each power $s^k$ vanish identically.
We can re-write~\eqref{eq:condser} as the solution of a linear system with an 
infinite matrix $M$ so that $(M-\textnormal{Id})v=0$. It is easy to see
that $M$ is computable recursively from $\pi_j$; indeed, if we let 
\be
\label{eq:psequat}
\Pi(s)^k=\left(\sum_{j=0}^\I\pi_j s^j\right)^k=:\sum_{j=0}^\I \pi_{j,k}s^j
\ee
then $\pi_{0,k}=\pi_0^k$ and $\pi_{j,k}=\frac{1}{j\pi_0}\sum_{l=1}^j[(k+1)l-j]
\pi_l\pi_{j-l,k}$~which follows from differentiating \eqref{eq:psequat} with respect 
to $s$ and re-arranging terms.
\end{proof}

It may seem that computing $v$ by solving~\eqref{eq:linsyst} may also be 
sufficient since one can just use the coefficients $v_j$ to get the required 
analytic eigenvalue function $a_1(s)$ using the same trick as in the proof of 
Theorem~\ref{thm:buffer1} replacing $\txte^{-ay}y$ in~\eqref{eq:Hopf_nf1}
by a more general analytic function of $y$. However,  the problem is that
$\sum_{k=0}^\I v_ks^k$ may not be a convergent power series on the required 
domain of definition.

At first, it may seem natural to adopt an operator-theoretic viewpoint 
for $M$. Let $w:\N\ra (0,+\I)$ be a weight function and consider
\be
l^p(\N,w):=\left\{v=(v_k)_{k\in\N}:\|v\|_{p,w}:=
\left(\sum_{k=1}^\I |v_k|^p w(k)\right)^{1/p}<\I\right\}
\ee
for $p\in[1,+\I)$. The next result shows why an operator-theoretic viewpoint 
leads to substantial complications for arbitrary analytic maps $\Pi$.

\begin{lem}
Given any weight function $w$ such that $(w(k))_{k\in \N}\in l^1(\N,1)$ and 
fix any $p\in[1,+\I)$, then there exists an analytic map $\Pi$ such that $M$ 
does not map $l^p(\N,w)$ into itself.
\end{lem}

\begin{proof}
Observe that the first row of $M$ is given by 
\benn
M_{1,\cdot}=(\pi_2,\ldots,
(M_{1,k-1}+k-1)\pi^{k-2}_0\pi_1^2+k\pi_0^{k-1}\pi_2,
(M_{1,k}+k)\pi_0^{k-1}+(k+1)\pi_0^k\pi_2,\ldots).
\eenn
Now select $v=(1,1,1,\ldots)$ which is clearly in $l^p(\N,w)$ 
as $(w(k))_{k\in \N}\in l^1(\N,1)$. Therefore, we have
\benn
(Mv)_1=\sum_{k=1}^\I M_{1,k}
\eenn
which diverges if we select suitable a suitable map $\Pi$, say for example 
$\pi_{0,1,2}=1$, and so $(Mv)\not\in l^p(\N,w)$.
\end{proof}

Similar results hold for other function spaces, i.e., $M$ is not tractable using the
classical theory of bounded operators if $\Pi$ is arbitrary. Many different 
restrictions for $\Pi$ are possible but a 
natural assumption is $\pi_0=0$ as this corresponds to the condition $\Pi(0)=0$ 
which should be imposed in the limit 
$\kappa_\mu\ra 0$, i.e., if the support of $p_0$ limits onto $s=0$, since any
vector field maps the initial condition to itself if no time has elapsed.

\begin{lem}
\label{lem:trdiag}
If $\pi_0=0$ then $M$ is lower-diagonal with entries $m_{ii}=\pi_1^{i+1}$.  
\end{lem}

\begin{proof}
Recall that $M$ is constructed by collecting terms of different orders of $s^l$ for 
$l\in\{2,3,4,\ldots\}$ from the expression
\benn
\sum_{k=2}^\I v_{k-1}\left(\sum_{j=0}^\I\pi_j s^j\right)^{k}=
\sum_{k=1}^\I v_{k}\left(\sum_{j=1}^\I\pi_j s^j\right)^{k+1}
\eenn
where we have used $\pi_0=0$. In particular, $m_{ij}$ is can be nonzero if and only 
if $v_j$ appears in the term collected for order $l=i+1$. Therefore, fixing any $i$ 
the largest index where a possible nonzero entry $m_{ij}$ occurs is for $j=i$. In fact,
it is also easy to see that the only term arising on the diagonal is $\pi_1^{i+1}$. 
\end{proof}

Even though $M$ is lower-diagonal, its entries still grow via certain multinomial coefficients.
Therefore, $M$ is not be a bounded operator for many $\Pi$. The following result 
even shows that it will be impossible to find an exact 
solution $v$ in many cases without imposing additional conditions.

\begin{prop}
If $\pi_0=0$ and $\pi_1\neq\pm1$, then $(M-\textnormal{Id})v=0$ if and only if 
$v=(0,0,0,\ldots)^\top$.
\end{prop}

\begin{proof}
From Lemma~\ref{lem:trdiag} we know that $m_{11}=\pi_1^2$. The first row
of $(M-\Id)v=0$ yields 
\benn
0=v_1 (m_{11}-1)=v_1(\pi_1^2-1)\quad \Leftrightarrow\quad v_1=0
\eenn
as $\pi_1^2\neq 1$. A direct process by mathematical induction on the 
rows yields the result.
\end{proof}

In general, it is no problem to impose the condition $\pi_1=1$ or $\pi_1=-1$
as we only want to match a pair of densities $p_0,p_*$ via $\Pi$ which have 
supports outside of the compact set $[-\kappa_\mu,\Pi(-\kappa_\mu)]$ containing 
the origin. However, imposing both conditions $\pi_0=0$ and $\pi_1=\pm1$ does
not work for nonlinear maps $\Pi$.

\begin{prop}
\label{prop:impossible}
If $\pi_0=0$, $\pi_1=\pm1$ and $v$ is a nontrivial solution then $\pi_k=0$ for 
all $k\geq2$. 
\end{prop}

\begin{proof}
Since $v$ is nontrivial, there exists some $j$ such that $v_k=0$
for all $k<j$ and $v_j\neq 0$. Consider the entry $m_{(j+1)j}$ and observe 
that it can be explicitly calculated 
\benn
m_{(j+1)j}=(j+1)\pi_2. 
\eenn
Since $v_k=0$ for all $k<j$ we must solve the
following equation coming from the $(j+1)$-row of $(M-\Id)v=0$
\benn
(j+1)\pi_2 v_j+v_{j+1}(\pi_1^{j+2}-1)=0.
\eenn
Since $\pi_1=1$ or $\pi_1=-1$ it follows that $\pi_2=0$ as $v_j\neq 0$.
By induction on the minor diagonals of the matrix we see that $\pi_2=0$ then
implies $\pi_3=0$ and so on.
\end{proof}

\textbf{Acknowledgements:} I would like to thank the European Commission (EC/REA) 
for support by a Marie-Curie International Re-integration Grant. 
I also acknowledge support via an APART fellowship of the Austrian Academy of 
Sciences ({\"{O}AW}).

\bibliographystyle{plain}
\bibliography{../my_refs}

\begin{thebibliography}{10}

\bibitem{BaerErneuxRinzel}
S.M. Baer, T.~Erneux, and J.~Rinzel.
\newblock The slow passage through a {Hopf} bifurcation: delay, memory effects,
  and resonance.
\newblock {\em SIAM J. Appl. Math.}, 49(1):55--71, 1989.

\bibitem{BaerGaekel}
S.M. Baer and E.M. Gaekel.
\newblock Slow acceleration and deacceleration through a {H}opf bifurcation:
  power ramps, target nucleation, and elliptic bursting.
\newblock {\em Phys. Rev. E}, 78:036205, 2008.

\bibitem{Benoit8}
E.~Beno{\^{i}}t, editor.
\newblock {\em Dynamic Bifurcations}, volume 1493 of {\em Lecture Notes in
  Mathematics}.
\newblock Springer, 1991.

\bibitem{BerglundGentz}
N.~Berglund and B.~Gentz.
\newblock {\em Noise-Induced Phenomena in Slow-Fast Dynamical Systems}.
\newblock Springer, 2006.

\bibitem{ChenKololnikovTzouGai}
Y.~Chen, T.~Kolokolnikov, J.~Tzou, and C.~Gai.
\newblock Patterned vegetation, tipping points, and the rate of climate change.
\newblock {\em Eur. J. Appl. Math.}, pages 1--14, 2015.

\bibitem{CoddingtonLevinson}
E.A. Coddington and N.~Levinson.
\newblock {\em Theory of Ordinary Differential Equations}.
\newblock McGraw-Hill, 1955.

\bibitem{VandenBroeckMandel}
C.~Van den Broeck and P.~Mandel.
\newblock Delayed bifurcations in the presence of noise.
\newblock {\em Phys. Lett. A}, 122:36--38, 1987.

\bibitem{DumortierRoussarie}
F.~Dumortier and R.~Roussarie.
\newblock {\em Canard Cycles and Center Manifolds}, volume 121 of {\em Memoirs
  Amer. Math. Soc.}
\newblock AMS, 1996.

\bibitem{Fenichel4}
N.~Fenichel.
\newblock Geometric singular perturbation theory for ordinary differential
  equations.
\newblock {\em J. Differential Equat.}, 31:53--98, 1979.

\bibitem{FruchardSchaefke}
A.~Fruchard and R.~Sch{\"{a}fke}.
\newblock A survey of some results on overstability and bifurcation delay.
\newblock {\em Discr. Cont. Dyn. Sys. - Series S}, 2(4):941--965, 2009.

\bibitem{Haberman}
R.~Haberman.
\newblock Slowly varying jump and transition phenomena associated with
  algebraic bifurcation problems.
\newblock {\em SIAM J. Appl. Math.}, 37(1):69--106, 1979.

\bibitem{HoldenErneux}
L.~Holden and T.~Erneux.
\newblock Slow passage through a {Hopf} bifurcation: from oscillatory to steady
  state solutions.
\newblock {\em SIAM J. Appl. Math.}, 53(4):1045--1058, 1993.

\bibitem{JansonsLythe}
K.M. Jansons and G.D. Lythe.
\newblock Stochastic calculus: application to dynamic bifurcations and
  threshold crossings.
\newblock {\em J. Stat. Phys.}, 90:227--251, 1998.

\bibitem{Jones}
C.K.R.T. Jones.
\newblock Geometric singular perturbation theory.
\newblock In {\em Dynamical Systems (Montecatini Terme, 1994)}, volume 1609 of
  {\em Lect. Notes Math.}, pages 44--118. Springer, 1995.

\bibitem{KruSzm1}
M.~Krupa and P.~Szmolyan.
\newblock Geometric analysis of the singularly perturbed fold.
\newblock {\em in: Multiple-Time-Scale Dynamical Systems}, IMA Vol.
  122:89--116, 2001.

\bibitem{KuehnBook}
C.~Kuehn.
\newblock {\em Multiple Time Scale Dynamics}.
\newblock Springer, 2015.
\newblock 814 pp.

\bibitem{Kuske}
R.~Kuske.
\newblock Probability densities for noisy delay bifurcation.
\newblock {\em J. Stat. Phys.}, 96(3):797--816, 1999.

\bibitem{Kuske1}
R.~Kuske.
\newblock Gradient-particle solutions of {Fokker-Planck} equations for noisy
  delay bifurcations.
\newblock {\em SIAM J. Sci. Comput.}, 22(1):351--367, 2000.

\bibitem{DeMaesschalck}
P.~De Maesschalck.
\newblock On maximum bifurcation delay in real planar singularly perturbed
  vector fields.
\newblock {\em Nonlinear Analysis}, 68:547--576, 2008.

\bibitem{Neishtadt1}
A.I. Neishtadt.
\newblock Persistence of stability loss for dynamical bifurcations. {I}.
\newblock {\em Differential Equations Translations}, 23:1385--1391, 1987.

\bibitem{Neishtadt2}
A.I. Neishtadt.
\newblock Persistence of stability loss for dynamical bifurcations. {II}.
\newblock {\em Differential Equations Translations}, 24:171--176, 1988.

\bibitem{Schecter}
S.~Schecter.
\newblock Persistent unstable equilibria and closed orbits of a singularly
  perturbed equation.
\newblock {\em J. Differential Equat.}, 60:131--141, 1985.

\bibitem{Shishkova}
M.A. Shishkova.
\newblock Analysis of a system of differential equations with a small parameter
  at the higher derivatives.
\newblock {\em Akademiia Nauk SSSR, Doklady}, 209:576--579, 1973.

\bibitem{Sontag}
E.D. Sontag.
\newblock {\em Mathematical Control Theory}.
\newblock Springer, 2nd edition, 1998.

\bibitem{Stiefenhofer2}
M.~Stiefenhofer.
\newblock Singular perturbation with {Hopf} points in the fast dynamics.
\newblock {\em Z. Angew. Math. Phys.}, 49(4):602--629, 1998.

\bibitem{Strand}
J.L. Strand.
\newblock Random ordinary differential equations.
\newblock {\em J. Differential Eqaut..}, 7(3):538--553, 1970.

\bibitem{Su}
J.~Su.
\newblock The phenomenon of delayed bifurcation and its analysis.
\newblock In C.K.R.T. Jones, editor, {\em Multiple Time Scale Dynamical
  Systems}, volume 122 of {\em The IMA Volumes in Mathematics and its
  Applications}, pages 203--214. Springer, 2001.

\bibitem{TzouWardKolokolnikov}
J.C. Tzou, M.J. Ward, and T.~Kolokolnikov.
\newblock Slowly varying control parameters, delayed bifurcations, and the
  stability of spikes in reaction-diffusion systems.
\newblock {\em Physica D}, 290:24--43, 2015.

\bibitem{ZeghlacheMandelVandenBroeck}
H.~Zeghlache, P.~Mandel, and C.~Van den Broeck.
\newblock Influence of noise on delayed bifurcations.
\newblock {\em Phys. Rev. A}, 40:286--294, 1989.

\end{thebibliography}

\end{document}